\documentclass[11pt,reqno,a4paper]{amsart}

\usepackage[latin1]{inputenc}
\usepackage{tikz}
\usetikzlibrary{shapes,arrows}
\usetikzlibrary{arrows.meta}
\usepackage{forest}
\usepackage{tikz-qtree}

\usetikzlibrary{arrows,shapes,positioning,shadows,trees}

\usepackage[hidelinks]{hyperref}

\usepackage{etoolbox}
\usepackage{amsmath}
\usepackage{enumerate}
\usepackage{amssymb}
\usepackage{amscd}
\usepackage{amsthm}
\usepackage{amsfonts}
\usepackage{graphicx}
\usepackage[all,cmtip]{xy}
\usepackage{enumitem}

\patchcmd{\subsection}{-.5em}{.5em}{}{}
\patchcmd{\subsubsection}{-.5em}{.5em}{}{}

\usepackage{enumitem}

\makeatother
\usepackage{hyperref}

\numberwithin{equation}{section}

\newcommand{\Y}{Y}

\newcommand{\y}{y}

\newcommand{\X}{X}

\newcommand{\x}{x}

\newcommand{\G}{G}
\newcommand{\g}{g}
\newcommand{\K}{K}
\newcommand{\kk}{k}
\newcommand{\Lam}{\Lambda}
\newcommand{\la}{\lambda}
\newcommand{\Gam}{\Gamma}
\newcommand{\ga}{\gamma}
\newcommand{\E}{E}

\newcommand{\R}{R}
\newcommand{\rr}{r}
\newcommand{\Ss}{S}
\newcommand{\s}{s}
\newcommand{\U}{U}
\newcommand{\uu}{u}
\newcommand{\V}{V}
\newcommand{\vv}{v}

\newcommand{\Cc}{C}

\newcommand{\W}{W}
\newcommand{\ww}{w}

\newcommand{\ttt}{t}

\newcommand{\Z}{Z}
\newcommand{\cZ}{\mathcal{Z}}

\newcommand{\z}{z}

\newcommand{\nl}{l}
\newcommand{\J}{J}

\newcommand{\m}{m}
\newcommand{\n}{n}

\newcommand\subsetsim{\mathrel{%
\ooalign{\raise0.2ex\hbox{$\subset$}\cr\hidewidth\raise-0.8ex\hbox{\wwcalebox{0.9}{$\wim$}}\hidewidth\cr}}}

\DeclareMathOperator{\Ent}{Ent}
\DeclareMathOperator{\BndEnt}{BndEnt}
\DeclareMathOperator{\Prob}{Prob}
\DeclareMathOperator{\SubSum}{SubSum}

\DeclareMathOperator{\pr}{pr}

\DeclareMathOperator{\poi}{Poi}
\DeclareMathOperator{\supp}{supp}

\theoremstyle{theorem}
\newtheorem{theorem}{Theorem}[section]

\newtheorem{corollary}[theorem]{Corollary}
\newtheorem{proposition}[theorem]{Proposition}
\newtheorem{lemma}[theorem]{Lemma}

\theoremstyle{definition}
\newtheorem{definition}[theorem]{Definition}
\newtheorem{remark}[theorem]{Remark}

\tikzstyle{decision} = [diamond, draw, fill=blue!20, 
    text width=4.5em, text badly centered, node distance=3cm, inner sep=0pt]
\tikzstyle{block} = [rectangle, draw, fill=blue!20, 
    text width=5em, text centered, rounded corners, minimum height=2em]
\tikzstyle{line} = [draw, -latex']
\tikzstyle{cloud} = [draw, ellipse,fill=red!20, node distance=3cm,
    minimum height=2em]

\renewcommand\labelenumi{(\roman{enumi})}
\renewcommand\theenumi\labelenumi

\usepackage{changepage}
\usepackage{mathtools}
\usepackage{graphicx}
\usepackage{calc}

\DeclarePairedDelimiterX\Set[2]{\{}{\}}{#1\,\delimsize\vert\,#2}

\begin{document}
\bibliographystyle{plain} 

\title{Boundary entropy spectra as finite subsums} 

\author{Hanna Oppelmayer}

\address{previous: Chalmers, Sweden\\current: TU Graz, Austria}
\email{oppelmayer@math.tugraz.at}



\maketitle

\begin{abstract}\noindent
In this paper we provide a concrete construction of  Furstenberg entropy values of $\tau$-boundaries of the  group $\mathbb{Z}[\frac{1}{p_1},\ldots,\frac{1}{p_{\nl}}]\rtimes  \{p_1^{\n_1}\cdots p_{\nl}^{\n_{\nl}} \, : \, \n_i\in\mathbb{Z}\}$ by choosing an appropriate random walk $\tau$.
We show that the boundary entropy spectrum can be realized as the subsum-set for any given finite sequence of positive numbers.
\noindent 
\end{abstract}

\section{Introduction}
A \textit{random walk} on a countable infinite discrete group $\Gamma$ is given by a probability measure $\tau$ on  $\Gamma$, whose support generates the group as a semi-group.    In  1963 Furstenberg introduced  in \cite{furst} so-called \textit{boundaries} for random walks. These are standard probability spaces $(\Y,\eta)$ on which $\Gamma$ acts measurably such that $\eta$ is  $\tau$-stationary (see page 3 for a definition) and
for $\tau^{\bigotimes\mathbb{N}}$-almost every sequence $(\gamma_i)_{i\in\mathbb{N}}\in \Gamma^{\mathbb{N}}$ the image measures $\gamma_1\cdots\gamma_n \eta$ converge  in the weak*-topology 
on a compact model to a Dirac measure (see Definition~\ref{def: poi}). 
To quantify these spaces, Furstenberg introduced 
 a notion of \textit{entropy} defined as
$$h_{\tau}(\Y,\eta):=\int_{\Gamma} \int_{\Y} -\log\Big(\frac{d\gamma^{-1}\eta}{d\eta}(\y)\Big)\, d\eta(\y)\, d\tau(\gamma).
$$ 
The boundary entropy spectrum and the entropy spectrum of $(\Gamma,\tau)$ are defined as
$$\BndEnt(\Gamma,\tau):=\{h_{\tau}(\Y,\eta)\, : \, (\Y,\eta) \text{ a $\tau$-boundary}\}
$$
and 
$$\Ent(\Gamma,\tau):=\{h_{\tau}(\Y,\eta)\, : \, (\Y,\eta) \text{ a $\tau$-stationary, ergodic probability $\Gamma$-space}\}.
$$
The latter has been intensively studied, e.g. Lewis Bowen proved in \cite{bowen} that the free group $\mathbb{F}_r$ for $r\geq 2$ generators equipped with the uniform measure $\upsilon$ has a \textit{full entropy realization}, i.e. $\Ent(\mathbb{F}_r, \upsilon)=[0,h(\mathbb{F}_r,\upsilon)]$ where $h(\Gamma,\tau):=\sum_{\gamma\in\Gamma}-\log(\tau(\gamma)) \tau(\gamma)$ denotes the \textit{random walk entropy}. The random walk entropy is always a natural upper bound for the entropy spectrum. 
Nevo  showed in \cite{nevo} that an infinite  countable group with property $(T)$ will always have a so-called \textit{entropy gap} for any generating probability measure on the group, that is there is $\epsilon>0$ (depending on the measure) such that any stationary ergodic space with entropy less than $\epsilon$ has entropy zero. 
 Examples of groups without an entropy gap are for instance finitely generated virtually free groups with generating probability  measures  which have a finite first moment, shown by Hartman-Tamuz  in \cite{YairTamuz}. 
  The question about the shape of $\Ent(\Gamma,\tau)$ is still widely open for many measured groups. 
 \\Here we shall investigate the shape of $\BndEnt(\Gamma,\tau)$ for a concrete  group $\Gamma$ where we choose different measures $\tau$ to design various  boundary entropy spectra.
   Note that $$\BndEnt(\Gamma,\tau)\subseteq \Ent(\Gamma,\tau).$$ 
 In our previous work \cite{BHO1}, joint with  Michael Bj\"orklund and Yair Hartman, we explicitly constructed random walks on a certain class of groups such that their boundary entropy spectra realize any given subsum sets. In particular, we constructed two measures $\tau_1$ and $\tau_2$ on the same  group $\Gamma$ such that $\BndEnt(\Gamma, \tau_1)=[0,\poi(\Gamma,\tau_1)]$ while $\BndEnt(\Gamma, \tau_2)$ is a Cantor set. The groups considered there were countable infinite direct sums of certain countable groups. In this short note we consider a group which is not a countable infinite direct sum and prove a certain subsum-realization for the boundary entropy spectrum using slightly different methods.\\
Given  finitely many distinct primes, $p_1,\ldots,p_l$, we  study the group, the semi direct product
 $$
 \Gamma:=\mathbb{Z}\Big[\frac{1}{p_1},\ldots,\frac{1}{p_{\nl}}\Big]\rtimes  \{p_1^{\n_1}\cdots p_{\nl}^{\n_{\nl}} \, : \, \n_i\in\mathbb{Z}\}.
 $$
\begin{theorem}\label{thm: main}
Let $\beta_{1},\ldots,\beta_{\nl}$ be positive real numbers. Then 
there exists a finitely supported generating probability measure $\tau$ on $\Gam$ 
such that 
$$
\BndEnt(\Gamma,\tau)=\Big\{\sum_{j\in \J}\beta_j\, : \, \J\subseteq \{1,\ldots, \nl\} \Big\}.
$$ 
\end{theorem}

\subsection{Organization of the paper and ingredients of the proof of Theorem~\ref{thm: main}} 
In Section~\ref{sec: obs}  we express the entropy of   probability measures
which are absolutely continuous w.r.t. $\sigma$-finite product measures as  sums of averaged \textit{information functions}  of the single measures of the product (see Section~2 and ~3 for details and definitions). Our proof uses a  result by Anosov in \cite{anosov}. Further in  Corollary~\ref{cor: factor} we will see that the boundaries in concern meet the requirements for these entropy formulae. This is done by passing to a homogeneous setting of a certain  locally compact, totally disconnected measured group by  techniques we have developed in  our previous work \cite{BHO1} with Michael Bj\"orklund and Yair Hartman.
We will apply a theorem by Brofferio in \cite{broff fin}, stating that  the ``maximal'' boundary - the so-called \textit{Poisson boundary} $\poi(\Gamma,\tau)$ (see Definition~\ref{def: poi}) -  of a finitely generated subgroup $\Gamma$ of the affine group  $\mathbb{Q}\rtimes \mathbb{Q}^*$ equipped with a finitely supported,  generating probability measure $\tau$ 
can be realized as  $$ \poi(\Gamma,\tau)=\Big(\prod_{\phi_p(\tau)<0} \mathbb{Q}_{p},\nu\Big)
$$
where $\nu$ is the uniquely $\tau$-stationary probability measure on that space and $\phi_p(\tau)$ denotes the \textit{$p$-drift} for $\tau$ (see Section~\ref{sec: ex} for a definition) running over all primes $p$ including infinity (with $\mathbb{Q}_{\infty}:=\mathbb{R}$). Note that the above product could be empty as well - e.g. if $\tau$ is symmetric -  in which case we would see the one-point-space with the Dirac measure.

\subsection{Acknowledgments} I am very grateful to my supervisor Michael Bj\"orklund for pointing out Anosov's statement in Lemma~\ref{lem: anosov} and sharing many  ideas which are crucial for the proofs in this paper. I am deeply thankful to Yair Hartman for supporting me throughout my mathematical career, inspiring me and helping me in enormous exciting and fruitful mathematical discussions.
The author was partially supported by Austrian Science Fund FWF: P31889-N35.

\section{Preliminaries}
Let $\G$ be a locally compact, second countable 
 group and fix a left-Haar
 measure $m_{\G}$ on $\G$. Throughout  let $\mu$ be a probability measure on the Borel $\sigma$-algebra of $\G$ which is \textit{admissible}, that is, it is \textit{generating} in the sense that the semi-group generated by the support of $\mu$ is dense in $\G$ and it is \textit{spread-out}, meaning there exists some convolution power of $\mu$ which is absolutely continuous w.r.t.  $m_{\G}$.
A standard probability space  $(\X,\nu)$ on which $\G$ acts jointly Borel measurably is called \textit{$\G$-space}. The action is  \textit{non-singular} if  $\g\nu$ and $\nu$ have the same null-sets for every $\g\in \G$. In particular, in this case the Radon-Nikodym derivative $\frac{d \g^{-1}\nu}{d\nu}$ exists.
 For these kind of actions Furstenberg \cite{furst} defined a notion of entropy 
\[ h_{\mu}(\X,\nu)=\int_{\G}\int_{\X} -\log\Big(\frac{d\g^{-1}\nu}{d\nu}(\x)\Big)\, d\nu(\x)\, d\mu(\g)
\]
provided that the \textit{information function} 
$$I_{\nu}:(\g,\x)\mapsto \log\Big(\frac{d\g^{-1}\nu}{d\nu}(\x)\Big) $$ lies in $L^1(\G\times \X, \mu\otimes \nu)$.
For our purposes we may extend the definition of the information function to $\sigma$-finite measures on $\X$.

A measure $\nu$ on a measurable $\G$-space $(\X,\mathcal{B})$ is called \textit{$\mu$-stationary} if $\nu=\mu*\nu$, i.e. $$ \nu(A)=\int_{\X} \nu(\g^{-1}A)\, d\mu(\g) $$
for all $A\in \mathcal{B}$.
Note that the assumption on $(\G,\mu)$ to be admissible implies that every $\mu$-stationary measure is in particular non-singular (see e.g. \cite[Lemma 1.2]{rig}).
Moreover, $\mu$-stationary probability measures give rise to measure-preserving transformations on the following product space. 
\begin{lemma}[{\cite[Proposition 1.3]{furm}}]\label{lem: m.p.}
Let $\nu$ be a $\mu$-stationary probability measure on $\X$. 
Then
$T: \G^{\mathbb{N}}\times \X$ given by  $T((\g_n)_{n\geq 1},\x):=((\g_n)_{n\geq 2}, \g_1\x)$ 
is a measurable map which preserves the measure $\mu^{\otimes\mathbb{N}}\otimes \nu$., i.e. $T(\mu^{\otimes\mathbb{N}}\otimes \nu)=\mu^{\otimes\mathbb{N}}\otimes \nu$.
\end{lemma}

The above introduced notion of entropy is of particular interest for the following spaces.

\begin{definition}\label{def: poi} A compact probability space $(\X,\nu)$ on which $\G$ acts continuously with $\nu$ being $\mu$-stationary is called a  \textit{compact $\mu$-boundary} if   $$
\g_1\cdots g_n \nu\to \delta_{\x((\g_i)_{i\geq 1})} \ \text{ for $\mu^{\otimes \mathbb{N}}$-almost every $(\g_i)_{i\geq 1} \in\G^{\mathbb{N}}$}
$$ in the weak*-topology, with $\x((\g_i)_{i\geq 1})\in\X$. 
A probability space which is measure-theoretical $\G$-equivariantly isomorphic to a compact  $\mu$-boundary is called a  \textit{$\mu$-boundary}.
The \textit{Poisson boundary} is the maximal such boundary in the sense that every other $\mu$-boundary is a \textit{$\G$-factor} of it, i.e. there exist $\G$-equivariant measure-theoretical factor maps from the Poisson boundary of $(\G,\mu)$ to all $\mu$-boundaries. (Its existence is provied by a result of Furstenberg, see e.g. \cite{furst}). We will write $\poi(\G,\mu)$ to denote the Poisson boundary of $(G,\mu)$.
\end{definition}

\section{Entropy observations}\label{sec: obs}
\subsection{Absolutely continuous measures provide equal information ``on average''}\label{subsec: info}
Let $\G$ be a locally compact, second countable group equipped with an admissible probability measure $\mu$, and let $(\X,\nu)$ be a $\G$-space. We say $\nu$ is \textit{absolutely continuous} w.r.t. a measure $\xi$ on $\X$ and write $\nu\ll \xi$, if every $\xi$-null-set is as well a $\nu$-null-set. 
\begin{proposition}\label{prop: I}
Let $\nu$ be a $\mu$-stationary probability measure on $\X$ and let $\xi$ be a non-singular $\sigma$-finite measure on $\X$. 
If $\nu\ll\xi$ 
and  if $I_{\nu}$ and $I_{\xi}$ are both in $ L^1(\mu\otimes \nu)$, then 
$$\int_{\G}\int_{\X} I_{\nu}(\g,\x) \, d\nu(\x) \, d\mu(\g) = \int_{\G}\int_{\X}I_{\xi}(\g,\x) \, d\nu(\x) \, d\mu(\g).$$
\end{proposition}

The proof of the above Proposition~\ref{prop: I} relies on the following old observation by Anosov \cite{anosov} and will be provided at the end of this subsection.

\begin{lemma}[{\cite[Theorem 1]{anosov}}]\label{lem: anosov}
Let $(\Z,\cZ,\zeta)$ be a probability space and $T:\Z\to \Z$ be a measure-preserving transformation on $\Z$, i.e. $\zeta\circ T^{-1}=\zeta$. Given a measurable function $f:\Z\to \mathbb{R}$ such that $f-f\circ T$ lies in $L^1(\Z,\zeta)$, then 
$$ \int_{\Z}(f-f\circ T)(z)\, d\zeta(\z)=0.
$$
 \end{lemma}

\begin{proof}
Let us denote $h:=f\circ T-f$. By Birkhoff's Ergodic Theorem  for $\zeta$-a.e. $\z\in \Z$ $$ \mathbb{E}[h\vert \mathcal{I}(T)](\z)=\lim\limits_{n\to\infty}\frac{1}{n}\sum_{k=0}^{n-1}h(T^{k}(\z)) = \lim\limits_{n\to\infty}\frac{1}{n}(f(T^{n}(\z)-f(\z))=  \lim\limits_{n\to\infty}\frac{1}{n} f(T^{n}(\z)),$$
where $\mathcal{I}(T)$ denotes the $\sigma$-algebra of  $T$-invariant sets in $\cZ$, and $ \mathbb{E}[h\vert \mathcal{I}(T)]$ stands for the conditional expectation of $h$ w.r.t. this $\sigma$-algebra.
Moreover, for every $\epsilon>0$ we see that $$\zeta\Big(\Big\{\z\in\Z\, : \, \Big\vert\frac{ f(T^{n}(\z))}{n}\Big\vert >\epsilon\Big\}\Big)=
\zeta\Big(\Big\{\z\in\Z\, : \, \Big\vert\frac{ f(\z)}{n}\Big\vert >\epsilon\Big\}\Big)
\to 0
$$
by $T$-invariance of $\zeta$. 
Hence
(by using Borel-Cantelli)
 there exists a subsequence $(n_l)_{l\in\mathbb{N}}$ such that $\frac{ f(T^{n_l}(\z))}{n}\to 0$ for $l\to\infty$, for $\zeta$-a.e. $\z\in \Z$. Thus, $$0=\mathbb{E}[h\vert \mathcal{I}(T)]$$ and in particular
$$0=\int_{\Z} h\, d\zeta
$$ since $\mathbb{E}[h\vert \mathcal{I}(T)]\in L^1(Z,\zeta)$.
\end{proof}

\begin{corollary}\label{cor: int zero} Let $\nu$ be a $\mu$-stationary probability measure on $\X$. For a measurable function 
$u: \X\to \mathbb{R}$  such that $f(\g,\x):=u(\x)-u(\g\x)$ belongs to $L^{1}(\G\times \X,\mu\otimes \nu)$ we have 
$$\int_{\G}\int_{\X} f(\g,\x)\, d\nu(\x)\,  d\mu(\g)=0.$$
\end{corollary}

\begin{proof} 
Since $\nu$ is $\mu$-stationary,  by Lemma~\ref{lem: m.p.} we can consider the  measure-preserving transformation  $T: \G^{\mathbb{N}}\times \X$ defined as $T((\g_n)_{n=1}^{\infty},\x):=((\g_n)_{n=2}^{\infty}, \g_1\x)$.
Let us set $\widetilde{u}:=u\circ \pr_{\X}$ where $\pr_{\X}: \G^{\mathbb{N}}\times \X\to \X$ is the projection to $\X$. Then clearly, for $\g_1=\g$ we can write $$u(\x)-u(\g\x)=\widetilde{u}(\x)- \widetilde{u}(T((g,g_2,g_3,\ldots) 
,\x)).$$ Now to apply Lemma~\ref{lem: anosov}, it is left to show that
$$\widetilde{u}- \widetilde{u}\circ T\in L^{1}(\G^{\mathbb{N}}\times \X, \mu^{\otimes \mathbb{N}}\otimes \nu).$$
Indeed,  $\widetilde{u}$ does only depend on the first coordinate of $\G$ and by assumption $u-u\circ g\in L^{1}(\G\times \X,\mu\otimes \nu)$, which ends the proof.
\end{proof}

\begin{proof}[Proof of Proposition~\ref{prop: I}]
First note that  $u:=\frac{d\nu}{d\xi}>0$, $\nu$-almost everywhere. So
$$I_{\nu}(\g,\x)=\log\Big( \frac{u(g(\x)) d\g^{-1}\xi}{ u(\x)d\xi}(\x)\Big)=-\log(u(\x)) + \log(u(\g(\x))  +I_{\xi}(\g,\x)
$$ for $\mu\otimes \nu$-a.e. $(\g,\x)\in \G\times \X$. 
By Corollary~\ref{cor: int zero} we see that 
$$ \int_{\G\times \X} \log(u(\x))-\log(u(\g(\x)) \, d  \mu\otimes \nu(\g,\x)=0
$$
because $\log\circ u - \log\circ u\circ g = I_{\nu}-I_{\xi}\in L^{1}(\G\times \X, \mu\otimes \nu)$ by assumption. Thus, we obtain 
$$ \int_{\G\times \X} I_{\nu} \, d  \mu\otimes \nu=\int_{\G\times \X} I_{\xi} \, d  \mu\otimes \nu.
$$
\end{proof}

\subsection{An entropy formula for certain measures}\label{subsec: formula}
In  case when the $\mu$-stationary probability measure $\nu$ is absolutely continuous w.r.t. to a product measure, we can rewrite the entropy as sums of averaged information functions of every single measure of the product.

\begin{proposition}\label{prop: entropy formula} Let $\X=\prod_{i=1}^{\nl} X_i 
$ and $\G$ act diagonally on $\X$ such that  $\nu$ is $\mu$-stationary. Given non-singular $\sigma$-finite measures $\xi_i$ on $\X_i$ such that $\nu \ll \bigotimes_{i=1 
}^{\nl} \xi_i$, then $$
h_{\mu}(\X, \nu)= \sum_{i=1}^{\nl} \int_{\G} \int_{X} 
I_{\xi_i}(\g,\x_i)  \, d\nu ((\x_n)_{n=1}^{\nl})\, d \mu(\g)
$$
provided $I_{\nu}, I_{\xi_i 
}\in L^1(\mu\otimes \nu)$ for all $i=1,\ldots, \nl$.\\
In particular, if for every $\g\in\G$, for every $i\in \{1,\ldots,\nl\}$ there exists a $\xi_i$-conull set $\X_i^{\prime}\subseteq \X_i$ such that  
 $\frac{d\g^{-1} \xi_i}{d\xi_i}(\x_i) = \Delta_i(\g)$ for all  $\x_i\in\X_i^{\prime}$, then 
$$h_{\mu}(\X, \nu)=\sum_{i 
=1}^{\nl}
\int_{\G} -\log(\Delta_i(g)) \,  d \mu(g). 
$$
\end{proposition}
The proof of the above proposition  follows directly from Proposition~\ref{prop: I} 
 and the next basic fact about Radon-Nikodym derivatives on product spaces. 

\begin{lemma}\label{lem: rn}
Let $\eta_i$ and $\zeta_i$ be $\sigma$-finite measures on $\Y$ and $\Z$, respectively, for $i=1,2$ such that $\eta_1\ll \eta_2$ and $\zeta_1\ll \zeta_2$. Then $$
\frac{d(\eta_1\otimes \zeta_1)}{d(\eta_2\otimes \zeta_2)}(\y,\z)=
 \frac{d\eta_1}{d \eta_2}(\y)  \frac{d\zeta_1}{d \zeta_2}(\z).
$$
\end{lemma}

\begin{proof}[Proof of Proposition~\ref{prop: entropy formula}]
By Proposition~\ref{prop: I} we know that $$
h_{\mu}(\X,\nu)=\int_{\G}\int_{\X}-\log \Big(\frac{ d \g^{-1}\bigotimes_{i=1}^{\nl}\xi_{i}}{d \bigotimes_{i=1}^{\nl}\xi_{i}} (\x)\Big) \, d\nu(\x)\, d\mu(\g).
$$ Since the action is diagonal we can apply  Lemma~\ref{lem: rn}  and obtain
$$h_{\mu}(\X,\nu)=\sum_{i=1}^{\nl} \int_{\G}\int_{\X}-\log \Big(\frac{ d \g^{-1}\xi_{i}}{d \xi_{i}} (\x_i)\Big) \, d\nu(\x)\, d\mu(\g).
$$
\end{proof}

\section{An application}\label{sec: ex}
 In this section we consider the group $$\Gamma=\mathbb{Z}\Big[\frac{1}{p_1},\ldots,\frac{1}{p_{\nl}}\Big]\rtimes \Ss$$ where $\Ss:=  \{p_1^{\n_1}\cdots p_{\nl}^{\n_{\nl}} \, : \, \n_i\in\mathbb{Z}\}$ for given distinct primes $p_1,\ldots,p_l$.
We shall denote by $\Vert \cdot \Vert_p$ the $p$-adic norm on $\mathbb{Q}$ for every prime number $p$.

Following \cite{broff fin} the \textit{$p$-drift} of a measure $\tau$ on $\Gamma$ is defined as
$$\phi_p(\tau):=\sum_{\s\in \Ss}\pr_2\tau(\s)\log(\Vert \s \Vert_p)
$$ for a prime $p$. 
We say that $\tau$ has \textit{negative drift} if its $p$-drift is negative for every $p\in\{p_1,\ldots,p_{\nl}\}$.

\begin{theorem}[{in \cite[Theorem 1]{broff fin}}]\label{thm: broff} Let $\tau$ be a finitely supported, generating probability measure  on $\Gamma=\mathbb{Z}[\frac{1}{p_1},\ldots,\frac{1}{p_{\nl}}]\rtimes  \{p_1^{\n_1}\cdots p_{\nl}^{\n_{\nl}} \, : \, \n_i\in\mathbb{Z}\}$  with negative drift. 
Then the Poisson boundary of $(\Gam,\tau)$ is given by $$ (\mathbb{Q}_{p_1}\times \ldots \times \mathbb{Q}_{p_{\nl}},\nu)
$$ where $\nu$ is the unique $\tau$-stationary probability measure on this space.
\end{theorem}
To prove  Theorem~\ref{thm: main} we will view $\mathbb{Q}_{p_1}\times \ldots \times \mathbb{Q}_{p_{\nl}}$ as a homogeneous space  w.r.t.  a completion of $\Gamma$. 
 We can diagonally embed $$\rho:\Gamma \hookrightarrow \G:= \mathbb{Q}_{p_1}\times \ldots \times \mathbb{Q}_{p_{\nl}} \rtimes \Ss
$$ where $\Ss$ denotes $\{p_1^{\n_1}\cdots p_{\nl}^{\n_{\nl}} \, : \, \n_i\in\mathbb{Z}\}$. 
Note that the above embedding has a dense image, since $\mathbb{Z}[\frac{1}{p_1},\ldots, \frac{1}{p_{\nl}}]$ is dense in $\mathbb{Q}_{p_i}$ for every $p_i\in\{p_1,\ldots, p_{\nl}\}$. Indeed,  by the Theorem of Strong Approximation (e.g. \cite[Theorem 5.8]{ram}), for every $a_i\in \mathbb{Q}_{p_i}$, $\forall \epsilon >0$ there exists $x\in \mathbb{Q}$ such that $\Vert x-a_i\Vert_{p_i}<\epsilon$ $\forall i=1,\ldots, \nl$ and $\Vert x\Vert_{p}\leq 1$ for all primes $p$ except $\{p_1,\ldots,p_{\nl}\}$. The latter property implies that $x\in \mathbb{Z}[\frac{1}{p_1},\ldots, \frac{1}{p_{\nl}}]$.
\\

Let us set $$ \X:=\mathbb{Q}_{p_1}\times \ldots \times \mathbb{Q}_{p_{\nl}}.$$ Then, naturally, $\G\curvearrowright \X 
$ 
 continuously 
and  $\X \cong \G/\hat{\Ss}$ for $\hat{\Ss}:= \{0\}\rtimes \Ss $ becomes a  homogeneous space w.r.t. the action of $\G$.
\begin{lemma}\label{lem: homo factor}
Let $\nu$ be a non-singular probability measure on $\X$. Then every $\G$-factor of $(\X,\nu)$ is of the form 
$$\Big(\prod_{j\in \J} \mathbb{Q}_{p_j}, \pi_{\J} \nu
\Big)
$$ for some index subset $\J \subseteq \{1,\ldots,\nl\}$, where  $\pi_{\J}:\X \to \prod_{j\in \J} \mathbb{Q}_{p_j}$ denotes the given factor map. 
\end{lemma}

\begin{proof}
Let $(\Y,\pi \nu)$ be a measurable $\G$-factor of $(\X,\nu)$ with surjective $\G$-equivariant factor map $\pi$. 
Since $\X=\G/(\{0\}\times \Ss)$,  there exists a closed subgroup $\Cc$ of $\G$ which contains $\{0\}\times \Ss$ such that $(\Y,\pi \nu)\cong (\G/\Cc, \phi \nu )$ as $\G$-spaces (mod $0$), where $\phi:\G/(\{0\}\times \Ss) \to \G/\Cc$ denotes the canonical factor map, confer e.g. \cite[Chapter 4, Section 2, Proposition 2.4 (b)]{margulis}. 
Now, if $\Cc$ is properly larger than $\{0\}\times \Ss$, then 
 there exists $(\rr,\s)\in \Cc$ with $\rr\neq (0,\ldots, 0)$ in $\X$, i.e. there is some $j\in\{0,\ldots,\nl\}$ such that the $j$-th coordinate $\rr_j$ of $\rr$ is non-zero in $\mathbb{Q}_{p_j}$, then since $(\widetilde{\s}\rr, 1)= (0,\widetilde{\s})(\rr,\s)(0,\s^{-1}\widetilde{\s}^{-1})\in \Cc$ for any $\widetilde{\s}\in \Ss$ we obtain that the $j$-th coordinate in $\Cc$ contains all elements of the group generated by $\{\widetilde{\s} \rr_j\, : \, \widetilde{\s}\in \Ss\}$ which is exactly $\rr_j \mathbb{Z}[\frac{1}{p_1},\ldots,\frac{1}{p_{\nl}}]$. Since $C$ is closed and  $\mathbb{Z}[\frac{1}{p_1},\ldots,\frac{1}{p_{\nl}}]$ is dense in $\mathbb{Q}_{p_j}$ we obtain that the $j$-th  coordinate of $C$ has to be whole $\mathbb{Q}_{p_j}$. The same argument goes through for any non-zero coordinate entry of $C$ such that in the end we obtain $Y\cong \G/(\prod_{j\in J^c} \mathbb{Q}_{p_j})\cong \prod_{j\in J} \mathbb{Q}_{p_j}$ as $\G$-spaces for some index set $J\subseteq \{1,\ldots,\nl\}$, such that the measure $\pi\nu$ on $\Y$ is transformed  to $\pr_{J}\nu$ on $\prod_{j\in J} \mathbb{Q}_{p_j}$. 
\end{proof}

Now, in order to relat the boundary theory of $(\Gamma,\tau)$ with those of $(\G,\mu)$ we need to 
 require some relations for the measures in charge. For this end let us consider the following subgroups
 $$\Lam:=\mathbb{Z}\rtimes \{1\} \leq \Gamma \text{ and } \K:= \mathbb{Z}_{p_1}\times \ldots \times \mathbb{Z}_{p_{\nl}} \rtimes \{1\}.$$
We remark that $(\Gamma,\Lambda)$ is a so-called reduced Hecke pair -  i.e. has finite $\Lam$-orbits on $\Gamma/\Lambda$ - and $(\G,\K)$ its Schlichting completion, see e.g. \cite{hecke} for  details and definitions. Important to notice is  that the coset spaces $\Gam/\Lam$ and $\G/\K$ are $\Gamma$-equivariant isomorphic. Such an isomorphism is for instance given by  $\theta: \Gam/\Lam \longrightarrow \G/\K$ with $\gamma\Lambda\mapsto \rho(\gamma)\K$, see \cite[Proposition 3.9]{hecke}. \\
Given a probability measure $\tau$ on $\Gam$ we want to construct a probability measure $\mu_{\tau}$ on $\G$ such that they coincide on the coset spaces $\Gam/\Lam$ and $\G/\K$,  in a sense that \begin{equation}\label{eq: coset eq} \tau(\ga \Lam)=\mu_{\tau}(\theta(\ga \Lam)),\ \forall \ga\Lam\in\Gam/\Lam.\end{equation}
Given $\tau\in \Prob(\Gam)$ we set 
\begin{equation}\label{eq: mu}\mu_{\tau}(f):=
\sum_{\gamma\Lambda\in \Gamma/\Lambda}\tau(\gamma\Lambda)\int_{\K} f(\rho(\gamma)\kk)\, d\m_{\K}(\kk)
\end{equation}
 for $f:\G\to \mathbb{R}$ measurable, 
where $\m_{\K}$ denotes the Haar measure on $\K$ with $\m_{\K}(\K)=1$. 
 The above is well-defined due to left-$\rho(\Lambda)$-invariance of $\m_{\K}$.
 Remark that equation~\ref{eq: coset eq} is indeed fulfilled in this construction.

Measures which are $\Lam$-invariant on the coset spaces $\Gamma/\Lambda$ provide a theory of relating $(\G,\mu_{\tau})$-stationary spaces to  $(\Gam,\tau)$-stationary ones. As in our perivious work \cite{BHO1} we may give  these measures a name. 

\begin{definition}
We call   a probability measure $\tau$ on $\Gam$ \textit{$\Lam$-absorbing} if 
$$\tau(\la \ga \Lam)=\tau(\ga \Lam), \ \forall \la\in \Lam
$$ for all $\ga\in \Gam$.
\end{definition}

Using these measures in  Brofferio's result in Theorem~\ref{thm: broff} we obtain now a homogeneous setting.
Let us fix a Haar measure $\m_{\mathbb{Q}_{p_i}}$ on $\mathbb{Q}_{p_i}$ and let $\mu_{\tau}$  denote the probability measure on $\G$ in equation \ref{eq: mu} corresponding to $\tau$.

\begin{corollary}\label{cor: poi G}
Let $\tau$ be a  finitely supported, generating probability measure  on $\Gamma$ which is $\Lam$-absorbing and has negative drift. Then 
 $$ \poi(\G,\mu_{\tau})=\poi(\Gamma,\tau)=\Big(\prod_{i=1}^{\nl}\mathbb{Q}_{p_i},\nu\Big)$$ where $\nu$ is  the unique $\tau$-stationary and unique $\mu_{\tau}$-stationary probability measure on  $\prod_{i=1}^{\nl}\mathbb{Q}_{p_i}$.  Moreover, $\nu\ll\bigotimes_{i=1}^{\nl} \m_{\mathbb{Q}_{p_i}}$. 
\end{corollary}

\begin{proof}
Let $\nu$ be the unique $\tau$-stationary probability measure given by Theorem~\ref{thm: broff}.
It is not hard to show that every $\mu_{\tau}$-stationary probaility measure is $\tau$-stationary (see \cite[Theorem 1.6 (P3)]{BHO1}). 
Thus, there can be at most one $\mu_{\tau}$-stationary probability measure on $\mathbb{Q}_{p_1}\times \ldots \times \mathbb{Q}_{p_{\nl}}$ and it has to coincide with $\nu$.
The existence of a $\mu_{\tau}$-stationary probability measures can be shown along the same lines as Brofferio's construction in \cite{broff fin} for $\tau$-stationary probability measures. 
Note that $\mu_{\tau}$ is generating for $\G$ since $\tau$ is generating $\Gamma$ (\cite[Theorem 1.6 (P2)]{BHO1}).
In particular, $\nu$ is $\G$-quasi-invariant. Since on homogeneous spaces there is a unique $\G$-quasi-invariant measure class (see e.g.  \cite[Theorem 5.19]{quantum}) we obtain that $\nu \ll \bigotimes_{i=1}^{\nl} \m_{\mathbb{Q}_{p_i}}$.
\end{proof}

In \cite[Corollary 4.15]{BHO1}, it is shown that every $\Gamma$-factor of a $\G$-ergodic space is ($\Gamma$-equivariant) isomorphic to a $\G$-factor. Thus, from  Lemma~\ref{lem: homo factor} and Corollary~\ref{cor: poi G} we can conclude
\begin{corollary}\label{cor: factor}
Let $\tau$ be a $\Lam$-absorbing, finitely supported, generating probability measure on $\Gamma$  with negative drift. Then all $\tau$-boundaries of $\Gamma$ are of the form
$$\Big(\prod_{j\in \J}\mathbb{Q}_{p_j},  \eta_{\J}\Big)$$ for some index set $\J\subseteq \{1,\ldots,\nl\}$ 
where $\Gamma$ acts on $\prod_{j\in \J}\mathbb{Q}_{p_j}$ via the embedding in $\G$ and $\eta_{J} \ll \bigotimes_{j\in\J} \m_{\mathbb{Q}_{p_j}}$.
\end{corollary}

The above Corollary~\ref{cor: poi G} together with the entropy observation in Proposition~\ref{prop: entropy formula} will yield  to a proof of  Theorem~\ref{thm: main}, provided we construct $\tau$ such that it is finitely supported, $\Lambda$-absorbing and has negative drift.  
Let us do this construction for more general semi-directs products in the next subsection.

\subsection{A construction of absorbing measures to design entropy}
Here $\R$ shall be a 
countable
ring with $1$ and $\Ss$ a multiplicative subgroup of $\R$.
Let us consider a 
Hecke pair $(\Gam, \Lam)$ of the form \begin{center} $\Gam=\R\rtimes \Ss$ and $\Lam=\U\rtimes \{e\}$ \end{center} where $\U$ is an additive subgroup of $\R$.
Note that $(\Gam,\Lam)$ is a Hecke pair iff $[\U \, : \, \s\U\cap \U]<\infty $ for every $\s\in \Ss$.

We will construct a finitely supported, generating and $\Lam$-absorbing probability measure $\tau$ on $\Gam$ that is $\Lam$-invariant and of the following form 
\begin{equation}\label{eq: split}
\tau(\rr, \s)=\kappa_{\s}(\rr)\iota(\s)
\end{equation}
for probability measures $\kappa_{\s}$ on $\R$ and $\iota$ on $\Ss$.
In order to obtain that  $\tau$ is  $\Lam$-absorbing, i.e.   
$$\tau((\uu,e)(\rr,\s)\Lam)=\tau((\rr,\s)\Lam), \ \forall \uu\in\U$$
we need
\begin{equation}\label{eq: absorb}
 \kappa_{\s}(\uu+\rr+\s\U)= \kappa_{\s}(\rr+\s\U), \ \forall \uu\in\U
\end{equation}
for every   $\rr\in \R$ and $\s\in \Ss$ (with $\iota(\s)\neq 0$). 
To this end let $\V_{\s}\subseteq \R$ be a set of representatives for $\U/(\s\U\cap \U)$ with $0\in \V_{\s}$, where the cosets are taken w.r.t. addition. By the Hecke assumption $\V_s$ is finite for every $\s\in\Ss$. Similar as in  \cite[Remark 1.15]{BHO1}, we may thus construct $\Lam$-absorbing probability measures in the following way.
\begin{lemma}\label{lem: constr absorb}
Let $\kappa$ be a 
 probability measure  on $\R$. Then 
$$\kappa_{\s}(\ttt):=\frac{1}{\vert \V_{\s}  \vert } \sum_{\vv_{\s}\in \V_{\s}}\kappa (\ttt+\vv_{\s}), \ \ttt\in\R
$$  is a probability measure on $\R$  fulfilling  equation \ref{eq: absorb} for every $\s\in \Ss$.
In particular, if $\Gamma$ is finitely generated, 
 then  
 there exists a finitely supported, generating, $\Lambda$-absorbing probability measure $\tau$ on $\Gamma$ of the form \ref{eq: split}. 
\end{lemma}
\begin{proof}
We need to show that $\kappa_{\s}$ is $\U$-invariant on every coset $\rr+\s\U$, hence that
$$
 \sum_{\vv_{\s}\in \V_{\s} } \kappa(\uu_0+\rr+\s\U+\vv_{\s})= 
 \sum_{\vv_{\s}\in \V_{\s} }\kappa(\rr+\s\U+\vv_{\s}) \text{ for every } \uu_0\in\U.
$$
Let  $\W_{\s}$ be a set of representatives of $\s\U/(\s\U\cap \U)$.
In particular,
$$\s\U=\bigsqcup_{\ww_{\s}\in \W_{\s} }(\ww_{\s}+ \s\U \cap \U )\ \text{ and } \ \U=\bigsqcup_{\vv_{\s}\in \V_{\s} }(\vv_{\s}+ \s\U \cap \U),$$
hence  $$\bigsqcup_{\vv_{\s}\in \V_{\s} } (\s\U+\vv_{\s})= 
\bigsqcup_{\vv_{\s}\in \V_{\s} } \bigsqcup_{\ww_{\s}\in \W_{\s} }(\ww_{\s}+ \s\U \cap \U +\vv_{\s})=
\bigsqcup_{\ww_{\s}\in \W_{\s} } (\ww_{\s}+\U).$$ This implies
 $\sum_{\vv_{\s}\in \V_{\s} } \kappa(\uu_0+\rr+\s\U+\vv_{\s})=
  \sum_{\ww_{\s}\in \W_{\s} }\kappa(\rr+\U+\ww_{\s}) =
 \sum_{\vv_{\s}\in \V_{\s} }\kappa(\rr+\s\U+\vv_{\s}),$ which ends the proof of the first statement.
 
It is left to show that we can choose $\iota$ and $\kappa$ such that $\tau$ given by equation~\ref{eq: split} is finitely supported and generating, whenever $\Gamma$ is finitely generated. 
Let $\E$ be a finite set of generators of $\Gamma$. We shall find finitely supported probability measures $\iota$ on $\Ss$ and $\kappa$  on $\R$
such that \begin{equation}\label{eq: supp tau}\supp(\tau) =\{(\rr,\s)\, : \, \rr\in \supp(\kappa_{\s}), \, \s\in \supp(\iota)\}  \supseteq \E.
\end{equation}
First, we choose a finitely supported probability measure $\iota$ on $\Ss$ such that \begin{equation}\label{eq: supp}
\supp(\iota)\supseteq \pr_{\Ss}(\E).\end{equation}
 Such a measure always exists,  for instance  the uniform measure on $ \pr_{\Ss}(\E)$. 
  Further, for every $\s\in \Ss$ we see that  $ \supp(\kappa_{\s})\supseteq \supp(\kappa)$ because $0\in V_{\s}$ by assumption. 
 Thus, to obtain \ref{eq: supp tau} it suffices to take $\kappa$ to be a finitely supported probability measure with support containing $\pr_{\R}(\E)$, for example the uniform probability measure on $\pr_{\R}(\E)$.

\end{proof}

Now let us go back to the case $$\R=\mathbb{Z}[\frac{1}{p_1},\ldots,\frac{1}{p_{\nl}}]\text{  and } \Ss=\{p_1^{\n_1}\cdots p_{\nl}^{\n_{\nl}} \, : \, \n_i\in\mathbb{Z}\}.$$ 

\begin{lemma}\label{lem: beta}
Given $\beta_1,\ldots,\beta_{\nl}>0$ we can find a finitely supported, generating, $\Lam$-absorbing probability measure $\tau$ on $\Gamma$
 such that $$\phi_{p_i}(\tau)=-\beta_i$$ for any $i=1,\ldots,\nl$.
\end{lemma}
\begin{proof}
We will construct $\tau$ as in equation \ref{eq: split}.
Recall that $(\Ss,\cdot)\cong (\mathbb{Z}^{\nl},+)$.
We  may thus set $\iota$ to be a product measure
$$\iota=\bigotimes_{i=1}^{\nl} \iota_i
$$
 with $\iota_i$ being probability measures on $\mathbb{Z}$. With $\tau$ as in equation \ref{eq: split} we thus obtain 
 $$\phi_{p_i}(\tau)=\sum_{n\in\mathbb{N}} \iota_{i}(n) n\log(p_i).
 $$ 
 Hence we shall find finitely supported, generating probability measures
  $\iota_i$ such that \begin{equation}\label{eq: alpha} \sum_{n\in\mathbb{Z}}  \iota_{i}(n)  n \log(p_i)=\frac{\beta_i}{\log(p_i)}=: \alpha. \end{equation}
 Let $\sigma$ be a finitely supported, generating probability measure on $\mathbb{Z}$ with non-zero mean $\mathbb{E}[\sigma]=\sum_{n\in\mathbb{Z}}\sigma(n) n$, for example $\sigma=\frac{3}{4}\delta_{1}+\frac{1}{4}\delta_{-1}$. 
 For every $i$ choose $N_i$ such that $\frac{\beta_i}{\mathbb{E}[\sigma]}<N_i$. By setting
  $$\iota_i:=\frac{\beta_{i}}{\mathbb{E}[\sigma] N_i} \sigma^{*N_i}+
  \Big(1-\frac{\beta_{i}}{\mathbb{E}[\sigma] N_i}\Big) \delta_0
 $$ 
  we obtain an instance of a generating, finitely supported probability measure fulfilling equation \ref{eq: alpha} since $\mathbb{E}[ \sigma^{*N_i}]=N_i  \mathbb{E}[ \sigma]$.\\
   Recall that in order to obtain  $\tau(\rr,\s):=\kappa_{\s}(\rr)\iota(\s)$  to be $\Lambda$-absorbing, finitely supported and generating, we have seen in the proof of Lemma~\ref{lem: constr absorb}, that the only constraint on $\iota$ is that  $\supp(\iota)\subseteq \pr_{\Ss}(\E)$ (equation \ref{eq: supp}) for a  finite generating set $\E$ of $\Gamma$. To guarantee this we can just choose the above numbers $N_i$ so large that $\prod_{i=1}^{\nl}\supp(\sigma^{*N_i})\supseteq \pr_{\Ss}(\E)$ which gives $\supp(\iota) \supseteq \pr_{\Ss}(\E)$ since $ \supp(\iota_i) =\supp(\sigma^{*N_i})\cup\{0\}$.
 \end{proof}

\subsection{Proof of Theorem~\ref{thm: main}} 

Let $\Gamma$, $\G$, $\Lambda$ as in Theorem~\ref{thm: main} and set  $\X:=\mathbb{Q}_{p_1}\times \ldots \times \mathbb{Q}_{p_{\nl}}$.

\begin{proof}[Proof of Theorem~\ref{thm: main}]

We start with the ansatz $$
\tau(\rr,\s):=\kappa_{\s}(\rr)\bigotimes_{i=1}^{\nl}\iota_i(\s)
$$
where we identify $(\Ss,\cdot)\cong (\mathbb{Z}^{\nl},+)$ to let $\iota=\bigotimes_{i=1}^{\nl}\iota_i$ be a product measure on $\Ss$ with probability measures $\iota_i$ on $\mathbb{Z}$. By Lemma~\ref{lem: constr absorb} we can choose probability measures $\iota_i$ on $\mathbb{Z}$ and $\kappa_s$  on $\mathbb{Z}[\frac{1}{p_1},\ldots,\frac{1}{p_{\nl}}]$, such that $\tau$ becomes a $\Lambda$-absorbing, finitely supported and  generating probability measure with the property that \begin{equation}\label{eq: drift beta} \phi_{p_i}(\tau)=-\beta_i, \ \forall i. 
\end{equation} In particular, the drift is negative and thus by Corollary~\ref{cor: factor} we know that every $\tau$-boundary is of the form 
$$(\Y_{\J}, \eta_{\J}):=(\prod_{j\in \J}\mathbb{Q}_{p_j},\pi_{\J}\nu)$$ for some index set $\J\subseteq \{1,\ldots,\nl\}$ for $\nu$ given by Theorem~\ref{thm: broff} with  $\pi_{\J}\nu\ll  \bigotimes_{j\in J} \m_{\mathbb{Q}_{p_j}}$.
Thus, by  Proposition~\ref{prop:  entropy formula} we obtain 
$$h_{\tau}(\Y_{\J}, \eta_{\J})=\sum_{j\in \J} \int_{ \Gam} -\log(\vert \s \vert_{p_j}) \, d \tau(\rr,\s)=
\sum_{j\in \J}  -\phi_{p_j}
$$
since  $$ \frac{d(\rr,\s)^{-1} \m_{\mathbb{Q}_{p_j}}}{d\m_{\mathbb{Q}_{p_j}}} =
\vert \s \vert_{p_j}. $$ 
Thus, by equation \ref{eq: drift beta} we obtain  $h_{\tau}(\Y_{\J}, \eta_{\J})=\sum_{j\in \J}  \beta_j$.
\end{proof}

\end{document}